\documentclass[a4paper,reqno]{amsart}
\usepackage[T1]{fontenc}
\usepackage[utf8x]{inputenc}
\usepackage[english]{babel}
\usepackage{times}
\usepackage{amsmath,amsfonts,amssymb,upgreek,comment,mathtools,mathrsfs,slashed}
\usepackage[11pt,curve,matrix,arrow,frame,graph]{xy}
\usepackage{graphicx}
\usepackage{hyperref}
\usepackage{enumerate}
\usepackage{geometry}
\usepackage{xcolor}
\usepackage{esint}

\usepackage{amsthm}

\theoremstyle{plain}
\newtheorem*{theorem*}{Theorem}
\newcounter{intro}

\newtheorem{thm}[intro]{Theorem}
\newtheorem{crl}[intro]{Corollary}

\newtheorem{theorem}{Theorem}[section]
\theoremstyle{definition}
\newtheorem{defi}[theorem]{Definition}
\newtheorem{lemm}[theorem]{Lemma}
\newtheorem{coro}[theorem]{Corollary}
\newtheorem{prop}[theorem]{Proposition}
\newtheorem{theo}[theorem]{Theorem}

\theoremstyle{definition}
\newtheorem{rem}[theorem]{Remark}
\newtheorem{rems}[theorem]{Remarks}

\newcommand{\cref}[1]{Corollary~\ref{#1}}

\newcommand{\lref}[1]{Lemma~\ref{#1}}
\newcommand{\pref}[1]{Proposition~\ref{#1}}

\newcommand{\tref}[1]{Theorem~\ref{#1}}

\newcommand{\R}{\ensuremath{\mathbb R}}
\newcommand{\N}{\ensuremath{\mathbb N}}
\newcommand{\Z}{\ensuremath{\mathbb Z}}
\newcommand{\C}{\ensuremath{\mathbb C}}
\newcommand{\bD}{\ensuremath{\mathbb D}}

\newcommand{\bS}{\ensuremath{\mathbb S}}

\newcommand{\cF}{\mathcal{F}}

\newcommand{\cU}{\mathcal{U}}

\newcommand{\mcC}{\mathscr{C}}
\newcommand{\mcW}{\mathscr{W}}

\newcommand{\Spn}{\mathbf{S}}
\newcommand{\Dirac}{\slashed{D}}
\newcommand{\Dirb}{\slashed{\partial}}
\newcommand{\Spin}{\slashed{\Sigma}}
\newcommand{\cl}{\mathrm{cl}}
\DeclareMathOperator{\un}{\mathbf{1}}
\newcommand{\ve}{\ensuremath{\varepsilon}}

\newcommand{\Ricm}{\ensuremath{\mbox{Ric}_{\mbox{\tiny{--}}}}}

\DeclareMathOperator{\bg}{ \overline{\emph{g}}}

\DeclareMathOperator{\dv}{\it d\nu}
\DeclareMathOperator{\dA}{\it dA}
\DeclareMathOperator{\vol}{\it \nu}

\DeclareMathOperator{\BE}{BE}

\newcommand{\Ric}{\ensuremath{\mbox{Ric}}}
\newcommand{\Ricci}{\ensuremath{\mbox{Ricci}}}
\newcommand{\Rm}{\ensuremath{\mbox{Rm}}}

\newcommand{\scal}{\ensuremath{\mbox{Scal}}}

\DeclareMathOperator{\tr}{tr}

\DeclareMathOperator{\ind}{ind}

\DeclareMathOperator{\supp}{supp}

\def\cH{\mathcal H}

\def\sch{Schr\"odinger }

\def\ra{\rangle}
\def\la{\langle}

\title[Open $3$-manifolds with non negative Ricci curvature in a  spectral or integral sense]{Open $3$-manifolds with non negative Ricci curvature in a  spectral or integral sense}

\author{Gilles Carron}
\address{G. Carron, Nantes Université, CNRS, Laboratoire de Mathématiques Jean Leray, LMJL, UMR 6629, F-44000 Nantes, France.} 
\email{Gilles.Carron@univ-nantes.fr}

\begin{document}
\maketitle

\begin{abstract} We show that if  a complete Riemannian $3-$manifold has $L^{\frac 32}-$ integrable Ricci curvature, satisfies a Sobolev inequality and has a non negative Ricci curvature in a spectral sense, then it  is diffeomorphic to  $\R^3$. 
 
\end{abstract}
\section{Introduction}
According to Cohn-Vossen and Huber, a complete Riemann surface $(M^2,g)$ with integrable Gaussian curvature and such that 
$$\int_M K_g \dA_g>-2\pi$$ is conformally equivalent to $\C$ or $\C^\star$. A very nice spectral counterpart has been obtained by Castillon \cite{Cast}: if $(M^2,g)$ is a complete Riemann surface for which there is some $\lambda>1/4$ such that the Schrödinger operator $\Delta_g+\lambda K_g$ is non negative:
$$\forall \varphi\in \mcC^\infty_c(M)\colon \int_M |d\varphi|^2\dA_g\ge \lambda\int_M \lambda K_g\,\varphi^2\dA_g$$ then $(M^2,g)$ is conformally equivalent to $\C$ or $\C^\star$.
There have been several attempts to generalize the first result in higher dimension (see for instance \cite{Chang_2000, CarronHerzlich_2002}). In this paper, we obtain a 3-dimensional  topological rigidity based on a spectral pinching hypothesis.

We introduce the following notation: if $(M,g)$ is a Riemannian manifold and $x\in M$, we let $\Ric(x)$ to be the lowest eigenvalue of the Ricci tensor at $x$ so that $\Ricci_x\ge \Ric(x)\, g_x$ in the sense of quadratic forms, and we let $\Ricm(x)$ be the negative part of $\Ric(x)$, that is $\Ricm(x)=0$ if $ \Ric(x)\ge 0$ and otherwise $\Ricm(x)=-\Ric(x)$. So that $\Ricci_x\ge -\Ricm(x)\, g_x$ in the sense of quadratic forms.

Our standing hypothesis for a complete Riemannian manifold $(M^3,g)$ are 
\begin{itemize}
\item that for some positive constant $\upgamma$, it satisfies the Sobolev inequality:
\begin{equation}\label{Sob}
\tag{Sob${}_\upgamma$}\hspace{1cm}\forall \varphi\in \mcC^\infty_c(M)\colon \upgamma \left(\int_M \varphi^6\dv_g\right)^{\frac 13}\le \int_M |d\varphi|^2\dv_g,\end{equation}
\item  and the $L^{\frac32}$ integrability of the negative part of the Ricci curvature.
\end{itemize}
Our first result is the following
\begin{thm}\label{3D-neg} Let $(M^3,g)$ be  a complete Riemannian manifold satisfying the Sobolev inequality \eqref{Sob} and such that $\Ricm\in L^{3/2}$. If for some $\lambda>2$, the Schrödinger operator $\Delta-\lambda\Ricm$ is non negative, then $M^3$ is diffeomorphic to $\R^3$.
\end{thm}
Using Hölder inegality, it is easy to check that when $\lambda\le \upgamma\left( \int_{M}(\Ricm)^{\frac32}\dv_g\right)^{-\frac13}$ then the  Schrödinger operator $\Delta-\lambda\Ricm$ is non negative, so that we obtained the following corollary:
\begin{crl} If  $(M^3,g)$ is  a complete Riemannian manifold satisfying the Sobolev inequality \eqref{Sob} and such that
$$\left( \int_{M}\left(\Ricm\right)^{\frac32}\dv_g\right)^{\frac13}<\upgamma/2$$ then $M^3$ is diffeomorphic to $\R^3$.
\end{crl}

Ricci flow and minimal surface techniques have enabled a complete understanding of the  topology of closed 3-manifolds carrying a Riemannian metric with non-negative Ricci or scalar curvature.  Also closed  Riemannian manifold $(M^3,g)$ such that the \sch operator $\Delta_g+\frac12 \Ric$ is non negative has been topological classified \cite{Bour_2017}. 

Concerning open complete Riemannian 3-manifolds, there are a topological classification of 
\begin{itemize} \item those with non-negative Ricci curvature, using minimal surface technics \cite{Schoen_1982,Anderson_1989,Liu_2012},
\item those with bounded geometry and uniformly positive scalar curvature using Ricci flow technics \cite{Bessi_res_2011}, recently using the technics of $\mu-$bubbles, J. Wang has removed the assumption on the boundedness of the geometry \cite{Wang_2025}. \end{itemize}

It is also known that contractible complete Riemannian $3-$manifolds with positive scalar curvature or non negative scalar curvature and bounded geometry are diffeomorphic to $\R^3$. For this purpose several tools have been used: elaborated index theory for the Dirac operators \cite{Chang_2009}, minimal surface \cite{Wang_2023,Wang_2024} or recently Inverse mean curvature flow \cite{Chodosh:2025aa}.

A series of very recent and interesting works using Ricci flows technics provides conditions under which a complete Riemannian $n$-manifolds with a Sobolev inequality and a small $L^{n/2}$ of the Riemann tensor is diffeomorphic to $\R^n$ \cite{Chan_2022,Chan_2024,Chan_2024a,Chau:2024aa,Martens:2024aa,Chan_2025,Martens_2025,Lee:2025aa}. The  lastest development in this direction is   the following \cite{Lee:2025aa}: For each $\upgamma,v_0>0$, there is some small  positive constant $\upvarepsilon(n,\upgamma,v_0)$ such that if $(M^n,g)$ is a complete Riemannian manifold satisfying the Sobolev inequality 
\begin{equation}\label{Sobn}
\tag{Sob${}_\upgamma$}\hspace{1cm}\forall \varphi\in \mcC^\infty_c(M)\colon \upgamma\left(\int_M \varphi^{\frac{2n}{n-2}}\dv_g\right)^{1-\frac 2n}\le \int_M |d\varphi|^2\dv_g,\end{equation} 
the upper volume bound:
$$\forall x\in M, \forall r>0\colon \vol_g(B(x,t))\le v_0r^n$$ 
and the small integral curvature concentration condition:
$$\int_M \| \Rm\|^{\frac n2} \dv_g<\upvarepsilon(n,\upgamma,v_0),$$ then $M$ is diffeomorphic to $\R^n$.

According to \cite{Chau:2024aa}, it is also possible to remove the assumption on the upper volume bound and assume in place that for some $p<n/2$ we have $\Rm\in L^p\cap L^\infty$. Our result is a kind of refinement of these results (in dimension $3$), in which the assumptions about volume growth  and control of the positive part of the Ricci tensor are removed.

It is possible to improve the threshold $\lambda=2$  in \tref{3D-neg} in adding a $L^{\frac32}$ integrability on the full Ricci tensor:
\begin{thm}\label{3D-full} Let $(M^3,g)$ be  a complete Riemannian manifold satisfying the Sobolev inequality \eqref{Sob} and such that $\Ricci \in L^{3/2}$. If for some $\lambda>1$, the Schrödinger operator $\Delta-\lambda\Ricm$ is non negative, then $M^3$ is diffeomorphic to $\R^3$.
\end{thm}
And its corollary:
\begin{crl} If  $(M^3,g)$ is  a complete Riemannian manifold satisfying the Sobolev inequality \eqref{Sob} and such that
$$\left( \int_{M}\left\|\Ricci\right\|^{\frac32}\dv_g\right)^{\frac13}<\upgamma$$ then $M^3$ is diffeomorphic to $\R^3$.
\end{crl}
Where the norm above $\|\Ricci\|$ is the operator norm of the Ricci tensor that is:  $$-\|\Ricci\|\, g\le \Ricci\le \|\Ricci\|\, g\text{ in the sense of quadratic forms.}$$

A crucial element in the proof of our results is that when there is some $\lambda>1$ such that the Schrödinger operator $\Delta_g-\lambda\Ricm$ is non negative, then one can find a positive function $u\in \mcC^2(M)$ such that $\Delta_g u-\lambda\Ricm u=0$. The conformal metric $\bg=u^{\frac 2 \lambda} g$ is then complete and the weighted Riemannian manifold $\left(M^3,\bg, d\mu=u^{\frac{2}{\lambda}}\dv_g\right)$ has non negative $(\lambda-1){-1}$-Bakry-Emery Ricci tensor, that is this weighted Riemannian manifold satisfies the Bakry-Emery  $\BE(0,N)$ curvature condition for $N=3+\frac{1}{\lambda-1}$ \cite{Catino_2024,Carron_2025}. The transformation $(g,\dv_g)\mapsto \left(u^{\frac{2}{\lambda}}\,g,u^{\frac{2}{\lambda}}\dv_g\right)$ is called time change in the theory of diffusion process. It has the remarkable property that it preserves harmonic and Green functions. During the proofs we have to navigate between the metrics $g$ and $\bg$ , their associated Riemannian  measures $\dv_g$ and $\dv_{\bg}$ and the measure $u^{\frac{2}{\lambda}}\dv_g$.

Our proof follows the general scheme introduced by Schoen and Yau \cite{Schoen_1982}:
\begin{enumerate}[i)]
\item  the first stage is to show that the universal cover of $M$ is contractible. This is a consequence of the fact that the  spectral hypothesis of \tref{3D-full} implies that $M$ and any of its covering have only one end.  It should be noted that the assumptions for verifying this step could be significantly improved, indeed a optimal spectral splitting theorem has been recently show (see \cite{Antonelli:2024aa,Catino:2024aa} and \cite{LI2006921} for earlier result) that a complete Riemannian manifold $(M^n,g)$ for which there is some $\lambda>(n-1)/4$ such that the \sch $\Delta_g-\lambda\Ric$ is non negative has only one end or splits as a Riemannian product.
\item the second stage is to show that $M$ is simply connected. It is  where the proofs of \tref{3D-neg} and \tref{3D-full} differs. In the proof of \tref{3D-neg}, we used a volume argument à la Milnor as developed by Anderson \cite{Anderson_1990} on the weighted Riemannian manifold $\left(M^3,\bg, u^{\frac{2}{\lambda}}\dv_g\right)$ and in the proof of \tref{3D-full}, we prove a vanishing of some Callias index for the Dirac operator of for the Dirac operator of $(M,\bg)$. This result  is a slight improvement of \cite{Carron_2002}.
\item The last stage is to prove that $M$ is simply connected as infinity. There are in the literature several tools to prove such a result: index theorem for the Dirac operator   \cite{Chang_2009}, minimal surface  \cite{Wang_2023,Wang_2024} or inverse mean curvature flow \cite{Chodosh:2025aa} . We have not been able to develop one of these techniques satisfactorily in our setting. Our proof relies on monotone quantities associated to the level sets of harmonic functions or Green's kernel.  Such monotone quantities has been introduced by Colding on manifold with non-negative Ricci curvature \cite{Colding_2012} and then they were developed for Riemannian $3$-manifold with non-negative scalar curvature (see for instance \cite{Stern,AMO,MW1,MW2,Colding:2025aa}).  We will show that the upper level set of Green functions provides an exhaustion by bounded open subset whose boundary are spherical. 
\end{enumerate}

The structure of the paper is as follows. In section 2, we  describe some properties of the geometry after time change and the influence of a Sobolev inequality. In section 3, we provide a vanishing result for the $\upalpha$-genus of the boundary of complete Riemannian manifold satisfying some Sobolev inequality and with non negative scalar curvature. In section 4, it is explained how one can exploit the aera type functionnal associated to the level set of Green function. And in the last section, we collect all these results to prove \tref{3D-neg} and \tref{3D-full}. In a short appendix, we collect several  analytical facts that are crucial in several proofs.

\section{Geometry after time change}

In all this section, we consider a complete Riemannian manifold $(M^n,g)$ of dimension $n\ge 3$ such that for some $\lambda>n-2$ the Schrödinger operator $\Delta_g-\lambda\Ricm$ is non negative:
$$\forall \varphi\in \mcC^\infty_c(M)\colon \int_M |d\varphi|^2_g\dv_g\ge \lambda  \int_M \Ricm\, \varphi^2\dv_g.$$
By \pref{Barta}, we can find a positive function $u\in \mcC^2(M)$ such that 
$$\Delta_gu-\lambda\Ricm u=0.$$ 
Then the conformal metric $$\bg=u^{\frac{2}{\lambda}}g$$ is complete  \cite[lemma 2.2]{Catino_2024} and  the weighted Riemannian manifold $\left(M,\bg, d\mu=u^{\frac{2}{\lambda}}\dv_g\right)$ satisfies the  Bakry-Emery $\BE(0,N)$ curvature condition for $N=n+\frac{(n-2)^2}{\lambda-(n-2)}$  \cite{Catino_2024,Carron_2025}.

In this section, we provided some old and new geometric and analytical properties of $\bg,\mu$ and of the gauge $u$. 
\subsection{About the metric $\bg$}
\subsubsection{} Firstly, the   Bakry-Emery $\BE(0,N)$ curvature condition implies a Bishop-Gromov volume comparison for the $\bg-$geodesic ball \cite[Corollary 2]{Qian}:
\begin{equation}\label{BG}
\forall x\in M, \forall r,R\in (0,+\infty) \text{ with } r<R\colon \frac{\mu\left( B_{\bg}(x,R)\right)}{\mu\left( B_{\bg}(x,r)\right)}\le \left(\frac R r\right)^N.
\end{equation}
\subsubsection{}This curvature condition also implies some Li-Yau's estimates for the heat kernel of the weighted Riemannian manifold $\left(M,\bg, d\mu\right)$ \cite[Theorem 3.2]{Qian_1995}.
We let $L=u^{-\frac 2\lambda}\Delta_g$ be the Laplacian of this weighted Riemannian manifold, so that
$$\forall\phi,\psi\in\mcC^\infty_c(M)\colon \int_M \la d\phi,d\psi\ra_{\bg} d\mu=\int_M L\phi\, \psi d\mu.$$
Let $\left\{ H(t,x,y)\right\}_{t>0,\ x,y\in M}$ be the Schwartz kernel of the self-adjoint operator $e^{-tL}$:
$$\forall\phi,\psi\in\mcC^\infty_c(M)\colon \int_M e^{-tL}\phi\ \psi d\mu=\int_{M\times M} H(t,x,y)\phi(y)\psi(y)d\mu(x)d\mu(y).$$
Then there is a constant $\gamma>0$ (depending only on $n$ and $\lambda$) such that for any $t>0$ and any $x,y\in M\colon$
\begin{equation}\label{LY}
\frac{1}{\gamma} \frac{e^{-\gamma \frac{d_{\bg}^2(x,y)}{t}}}{\mu\left( B_{\bg}(x,\sqrt{t})\right)}\le H(t,x,y)\le \gamma \frac{e^{- \frac{d_{\bg}^2(x,y)}{\gamma t}}}{\mu\left( B_{\bg}(x,\sqrt{t})\right)}.
\end{equation}
These Gaussian estimates imply the Liouville's property (see for instance \cite[Theorem 4.6]{Saloff-Coste}): 
\begin{coro}\label{|iouville} A non negative harmonic function on $\left(M,\bg, d\mu\right)$ is constant.
\end{coro}
\subsubsection{} Moreover the Dirichlet energy of $\left(M,\bg, d\mu\right)$ and of $\left(M,g, \dv_g\right)$ are the same
$$\forall\phi,\psi\in\mcC^\infty_c(M)\colon \int_M \la d\phi,d\psi\ra_{\bg} d\mu= \int_M \la d\phi,d\psi\ra_{g}  \dv_g.$$
Hence $\left(M,g, \dv_g\right)$ is non-parabolic  if and only if $\left(M,g, \dv_g\right)$ is non-parabolic.
Note that if the Ricci curvature of $g$ is not non negative everywhere (that is if $\Ricm$ is not identically zero) then  $\left(M,g, \dv_g\right)$  is non-parabolic.

Moreover when $\left(M,g, \dv_g\right)$ is non-parabolic, the two associated minimal positive Green kernel coincides. Indeed if $\phi$ is the minimal solution of $L\phi=f$ (where $f\ge 0$ has compact support) then
$\Delta \phi=u^{\frac2\lambda} f$ and if $G$ is the minimal Green kernel of $\left(M,g, \dv_g\right)$ then
$$\phi(x)=\int_M G(x,y) u^{\frac2\lambda}(y) f(y)\dv_g(y)=\int_M G(x,y)  f(y)d\mu(y).$$

\subsubsection{The scalar curvature of the metric $\bg$}
\begin{lemm}\label{scalarbarg}Under the hypothesis made in this section, the scalar curvature of the conformal metric $\bg =u^{\frac 2\lambda} g$ is non negative:
$$\scal_{\bg}\ge 0.$$
\end{lemm}
\proof If one lets $v=u^{\frac{n-2}{2\lambda}}$ then 
$$\scal_{\bg}\,\, v^{\frac{n+2}{n-2}}=\frac{4(n-1)}{n-2}\Delta_g v+\scal_{g} v.$$
But 
$$\Delta_g v=\left(\frac{n-2}{2\lambda}u\Delta_g u-\frac{n-2}{2\lambda}\left(\frac{n-2}{2\lambda}-1\right)|du|_g^2\right)u^{\frac{n-2}{2\lambda}-2},$$
so that 
$$\scal_{\bg}\,\, v^{\frac{n+2}{n-2}}= \left(2(n-1)\Ricm+\scal_{g}+\frac{2(n-1)}{\lambda}\left(1-\frac{n-2}{2\lambda}\right)\frac{|du|_g^2}{u^2}\right)v,
$$
or \begin{equation}\label{scalbg}
\scal_{\bg}\,\, u^{\frac{2}{\lambda}}= 2(n-1)\Ricm+\scal_{g}+\frac{2(n-1)}{\lambda}\left(1-\frac{n-2}{2\lambda}\right)\frac{|du|_g^2}{u^2}.
\end{equation}
One has $\frac{n-2}{2\lambda}\le 1$ and $\scal_g\ge -n\Ricm$ hence
$$\scal_{\bg}\,\, u^{\frac{2}{\lambda}}\ge 2(n-1)\Ricm+\scal_{g}
\ge (n-2)\Ricm v\ge 0.$$
\endproof
\subsection{Properties of the gauge}The next proposition will be crucial later:
\begin{prop}\label{Mukhenhoupt}Under the hypothesis made in this section, then for any $\alpha\in [0,1]$, $u^\alpha$ is a $A_1$-weighted on the weighted Riemannian manifold $\left(M,\bg, d\mu\right)$:
There is a constant $C$ depending only on $\lambda,n$ and $\alpha$ such that for any $x\in M$ and $R>0\colon$
$$\fint_{B_{\bg}(x,R)} u^\alpha d\mu\le C \inf_{B_{\bg}(x,R)}u^\alpha.$$
\end{prop}
\proof If $\Ricm=0$, then $u$ is a positive harmonic function on a complete Riemannian with non-negative Ricci curvature, hence the Liouville property (\cref{|iouville}) implies that  $u$ is constant and the result is true. 
Assume that $\Ricm$ is not identically zero, then $(M,g,\dv_g)$ is non-parabolic. Now let us fix a point $o\in M$. And for any $R>0$, let $u_R$ be the minimal positive solution of the equation $\Delta u_R=\lambda\Ricm u\un_{B(o,R)}$. That is 
$$u_R(x)=\int_{B(o,R)} G(x,y)\lambda\Ricm(y) u(y)\dv_g(y)$$ where
$G$ is the minimal Green kernel of $(M,g\dv_g)$. The function $u_R$ is also the monotone limit as $\rho\to+\infty$ of the solution of the Dirichlet boundary problem 
$$\begin{cases} \Delta u_{R,\rho}=\lambda\Ricm\, u\un_{B(o,R)}&\text{ on } B(o,\rho)\\
u_{R,\rho}=0&\text{ on } \partial B(o,\rho).\end{cases}$$
The maximum principle yields that for any $\rho$ we have $u\ge u_{R,\rho}$ on $B(o,\rho)$ hence we always have
$u\ge u_R$, hence $\underline{u}=\sup_{R} u_R$ satisfies:
$$u\ge \underline{u}\text{ and } \forall x\in M\colon \underline{u}(x)=\int_M G(x,y)\lambda\Ricm(y) u(y)\dv_g(y).$$
Hence $u-\underline{u}$ is a non-negative harmonic function on $(M,g,\dv_g)$. Hence $u-\underline{u}$ is a non-negative harmonic function on $(M,\bg,d\mu)$.
But $(M,\bg,d\mu)$ has non negative Bakry-Emery Ricci curvature hence such a non-negative harmonic function is necessary constant by \cref{|iouville}. And  there is a constant $c\ge 0$ such that 
$$\forall x\in M\colon u(x)=c+\int_M G(x,y)\lambda\Ricm(y) u(y)\dv_g(y).$$
Using that $$G(x,y)=\int_0^{+\infty} H(t,x,y)dt$$ and that  $\int_M H(t,x,y)d\mu(y)=1$, we easily deduce that $$\forall t>0\colon u\ge e^{-tL}u.$$ Hence for any $z\in B_{\bg}(x,R):$
\begin{align*}
u(z)&\ge \int_M H(R^2,z,y)u(y)d\mu(y)\\
&\ge \int_{B_{\bg}(x,R)} H(R^2,z,y)u(y)d\mu(y)\\
&\ge \frac{1}{\gamma} \frac{e^{-4\gamma  }}{\mu\left( B_{\bg}(z,R)\right)}\int_{B_{\bg}(x,R)} u(y)d\mu(y)\text{   using the Li-Yaus'estimate \eqref{LY} }\\
&\ge \frac{e^{-4\gamma  }}{\gamma} \frac{\mu\left(B_{\bg}(x,R)\right)}{\mu\left( B_{\bg}(z,R)\right)}\fint_{B_{\bg}(x,R)} u(y)d\mu(y)\\
&\ge \frac{e^{-4\gamma  }}{\gamma}  2^{-N}\frac{\mu\left(B_{\bg}(x,2R)\right)}{\mu\left( B_{\bg}(z,R)\right)}\fint_{B_{\bg}(x,R)} ud\mu\text{   using the the volume comparison  \eqref{BG}}\\
&\ge \frac{e^{-4\gamma  }}{\gamma}  2^{-N}\fint_{B_{\bg}(x,R)} ud\mu.
\end{align*}
\endproof
\begin{rem} Using \cite[Proposition 4.4]{Carron:2022aa}, we know that in fact $u^\alpha$ is a $A_1$-weighted for any $\alpha\in \left(0,N/(N-2)\right)$.
\end{rem}
\subsection{With a Sobolev inequality} We assume now that moreover $(M^n,g)$ satisfies the Sobolev inequality \eqref{Sobn}. Such a Sobolev inequality implies that $(M^n,g)$  is non parabolic and according to \cite[Proposition 1.14]{carronlambda}, its minimal Green kernel $G$ satisfies that for a constant $C(n,\upgamma)$:
\begin{equation}\label{volfakeball}
\forall x\in M, \forall t>0\colon\ \vol_g\left\{ y\in \colon G(x,y)>t\right\}\le C(n,\upgamma)\ t^{\frac{n}{n-2}}.
\end{equation}
Hence for any $p>n/(n-1)$ and any $x\in M$:
$$\int_{M\setminus B(x,1)} G^p(x,y)\dv_g(y)<+\infty.$$
The same proof provides the same result about the equilibrium potential of a compact set. Recall that if $K\subset M$ is compact set, then its equilibrium potential $h_K\colon M\setminus K\rightarrow (0,1]$ is  the monotone limit of the solutions of  Dirichlet boundary value problem
$$\begin{cases}
\Delta h_R=0&\text{ on } B(o,R)\setminus K\\
h_R=1&\text{ on }  K\\
h_R=0&\text{ on } B(o,R).
\end{cases}
$$
Then 
\begin{equation}\label{integrapot}\forall p>n/(n-1), h_K\in L^p.\end{equation}
As the matter of fact if one tests the Sobolev inequality on the function 
$$\phi_\tau(x)=\begin{cases} \tau&\text{ if } h_K(x)>\tau\\
h&\text{ if } h_K(x)\le\tau.\end{cases}$$ The Green formula yields that 
$$\int_M |d\phi_\tau|^2_g\dv_g=\tau\int_{ h_K=\tau} \frac{\partial h_K}{\partial \vec n}d\mathcal H^{n-1}_g=\tau\int_{\partial h_K=1} \frac{h_K}{\partial \vec n}d\mathcal H^{n-1}_g=\tau\int_M |dh_K|^2_g\dv_g.$$ and hence
$$\upgamma\tau^2 \left(\vol_g\{ h_K>\tau\}\right)^{1-\frac 2n}\le \tau \int_M |dh_K|^2_g\dv_g.$$

\subsubsection{Integral bound on the gradient of the Green kernel}

\begin{prop}\label{Greengradient} Let $(M^n,g)$ be a complete  Riemannian manifold, $n\ge 3$ satisfying the hypothesis made in this section and assume $(M^n,g)$ is non parabolique. Let $G(o,\cdot)$ be its minimal Green kernel with pole at $o\in M$ and introduce the fake distance to $o$ by\footnote{where  $(n-2)|\bS^{n-1}|=1/c_n$.}
$$G(o,x)=\frac{c_n}{b^{n-2}(x)}.$$  Then, there is a constant $C$ such that  for any $r\ge 1$:
$$\int_{b=r} \frac{|db|_g^2}{r^{n-1}}d\cH^{n-1}_g\le C r^{1-\frac{1}{N-2}}.$$ 
Hence $$\int_{\{b\ge 1\}} \frac{|db|_g^3}{b^{n+1}} \dv_g=\int_1^{+\infty} \left( \int_{b=r}\frac{|db|_g^2}{r^{n+1}}d\cH^{n-1}_g\right)dr<+\infty.$$ \end{prop}
\proof
We introduce the weighted Riemannian manifold $\left(M,\bg, d\mu=u^{\frac{2}{\lambda}}\dv_g\right)$. It satisfies the Bakry-Emery $\BE(0,N)$ curvature condition. We already have noticed that the Green kernel of $(M,g,\dv_g)$ and of $\left(M,\bg, d\mu=u^{\frac{2}{\lambda}}\dv_g\right)$ coincide. Hence according to \cite[Proposition 1.1]{Song_2014}, we know that there is a constant  $C$ such that when $b\ge 1$ then
$$|db|_{\bg} \le C b^{1-\frac{1}{N-2}};$$
So that when $r\ge 1$:
$$\int_{b=r} \frac{|db|_g^2}{r^{n-1}}d\cH^{n-1}_g=\int_{b=r} \frac{|db|_g\,|db|_{\bg} u^{\frac 1\lambda} }{r^{n-1}}d\cH^{n-1}_g\le Cr^{1-\frac{1}{N-2}} \int_{b=r} \frac{|db|_g\, u^{\frac 1\lambda} }{r^{n-1}}d\cH^{n-1}_g.$$
But the function $u^{\frac 1\lambda} $ is superharmonic hence the Green's formula implies that the function
$$r\mapsto \int_{b=r} \frac{|db|_g\, u^{\frac 1\lambda} }{r^{n-1}}d\cH^{n-1}_g$$ is non increasing \cite{Colding_2012}. Hence
for any $r>0$:
$$\int_{b=r} \frac{|db|_g\, u^{\frac 1\lambda} }{r^{n-1}}d\cH^{n-1}_g\le |\bS^{n-1}| \ u^{\frac 1\lambda}(o).$$
And we have obtained the estimate:
$$\forall r>1\colon \int_{b=r} \frac{|db|_g^2}{r^{n-1}}d\cH^{n-1}_g\le C|\bS^{n-1}| \ u^{\frac 1\lambda}(o)\ r^{1-\frac{1}{N-2}}.$$

\endproof

\subsubsection{$\mu$-volume of $\bg$ geodesic ball}
\begin{prop}\label{muvollower}  Let $(M^n,g)$ be a complete  Riemannian manifold, $n\ge 3$ satisfying the hypothesis made in this section and assume  moreover that it satisfies the Sobolev inequality \eqref{Sobn}.  There is a constant $\upsigma$ depending only on $n$ and $\lambda$ such that it one considers the weighted Riemannian manifold $\left(M,\bg=u^{\frac 2\lambda}g, d\mu\right)$, then for any $x\in M$ and any $R>0$:
$$\upsigma\,\upgamma^{\frac n2}   u^{-\frac{n-2}{\lambda}}(x) R^n \le \mu\left(B_{\bg}(x,R)\right).$$
\end{prop}
\proof The Sobolev inequality implies that 
$$\forall \varphi\in \mcC^\infty_c(M)\colon \upgamma\left(\int_M \varphi^{\frac{2n}{n-2}}\, u^{-\frac 2\lambda}d\mu \right)^{1-\frac 2n}\le \int_M |d\varphi|_{\bg}^2d\mu.$$
Applying this inequality for  $\varphi(y)=\left(R-d_{\bg}(x,y)\right)_+$, one gets that
$$\upgamma\frac{R^2}{4} \left(\mu\left(B_{\bg}(x,R/2) \right)\right)^{1-\frac2n}\, \left(\fint_{B_{\bg}(x,R/2)}  u^{-\frac 2\lambda}d\mu \right)^{1-\frac 2n}\le \mu\left(B_{\bg}(x,R) \right).$$
Hölder inequality yields that for any $p>1:$
$$1\le  \left(\fint_{B_{\bg}(x,R/2)}  u^{-\frac 2\lambda}d\mu \right) \left(\fint_{B_{\bg}(x,R/2)}  u^{\frac{ 2}{(p-1)\lambda}}d\mu \right)^{p-1}.$$
With $p=1+2/\lambda$ and using the $A_1$-property of $u$ (\pref{Mukhenhoupt}), one gets that 
$$1\le  \left(\fint_{B_{\bg}(x,R/2)}  u^{-\frac 2\lambda}d\mu \right) C \min_{B_{\bg}(x,R/2)} u^{\frac2\lambda}\le  \left(\fint_{B_{\bg}(x,R/2)}  u^{-\frac 2\lambda}d\mu \right)\, C u^{\frac2\lambda}(x).$$
So that 
$$\upgamma\frac{R^2}{4 C^{\frac{n-2}{n}}} u^{-\frac{2(n-2)}{n\lambda}}(x)\left(\mu\left(B_{\bg}(x,R/2) \right)\right)^{1-\frac2n}\le \mu\left(B_{\bg}(x,R) \right).$$
If one defines $$\Theta=\upgamma^{\frac n2}2^{\frac{n^2-4n}{2}}C^{\frac{n-2}{2}}u^{-\frac{n-2}{\lambda}}(x),$$ one has 
$$\left( \frac{\mu\left(B_{\bg}(x,R/2) \right)}{\Theta (R/2)^n}\right)^{1-\frac2n}\le \frac{\mu\left(B_{\bg}(x,R) \right)}{\Theta R^n}.$$
It is now easy to iterate this inequality and get that 
$$\Theta R^n\le \mu\left(B_{\bg}(x,R) \right).$$
\endproof

\subsubsection{Sobolev inequality  for $\bg$}
\begin{prop}\label{Sobbarg} Let $(M^n,g)$ a complete Riemannian manifold satisfying the hypothesis made in this section and assume moreover that the Sobolev inequality \eqref{Sobn} holds and that $\Ricci\in L^{\frac n2}$, then  there is a positive constant $\upupsilon$ such that the conformal metric $\bg =u^{\frac 2\lambda} g$satisfies the Sobolev inequality:
$$\forall \varphi\in \mcC^\infty_c(M)\colon \upupsilon \left(\int_M \varphi^{\frac{2n}{n-2}}\dv_{\bg}\right)^{1-\frac 2n}\le \int_M |d\varphi|_{\bg}^2\dv_{\bg}.$$
\end{prop}

\proof
Let $K\subset M$ be a compact set such that 
$$\frac{n-2}{4(n-1)}\left(\int_{M\setminus K} (\scal_g)_-^{\frac n2}\dv_g\right)^{\frac 2n}\le \frac{\upgamma}{2}$$ then using the Hölder inequality, one gets that 
\begin{align*}\forall \varphi\in& \mcC^\infty_c(M\setminus K)\colon\\
& \frac{\upgamma}{2}\left(\int_{M\setminus K} \varphi^{\frac{2n}{n-2}}\dv_{g}\right)^{1-\frac 2n}\le \int_{M\setminus K} \left[|d\varphi|_{g}^2+\frac{n-2}{4(n-1)}\scal_g \varphi^2\right]\ \dv_{g}.\end{align*}
That is the Yamabe constant of  $(M\setminus K, g)$ is positive, by conformal invariance of the Yamabe constant, one gets that 
\begin{align*}\forall \varphi\in& \mcC^\infty_c(M\setminus K)\colon\\
& \frac{\upgamma}{2}\left(\int_{M\setminus K} \varphi^{\frac{2n}{n-2}}\dv_{\bg}\right)^{1-\frac 2n}\le \int_{M\setminus K} \left[|d\varphi|_{\bg}^2+\frac{n-2}{4(n-1)}\scal_{\bg} \varphi^2\right]\ \dv_{\bg}.\end{align*}
Recall the formula \eqref{scalbg}:
$$\scal_{\bg}u^{\frac 2\lambda}=\scal_g+2(n-1)\Ricm+2(n-1)\left(1-\frac{n-1}{2\lambda}\right) \frac{ |du|^2_g}{u^2}.$$
Noticed that 
$$\left((\scal_g)_++2(n-1)\Ricm\right)^{\frac n2} u^{-\frac n\lambda}\dv_{\bg}=\left((\scal_g)_++2(n-1)\Ricm\right)^{\frac n2} \dv_{g}.$$
Hence if one considers a compact set $K_0$ containing $K$ with
$$\frac{n-2}{4(n-1)}\left( \int_{M\setminus K_0 } \left((\scal_g)_++2(n-1)\Ricm\right)^{\frac n2} \dv_{g}\right)^{\frac 2n}\le \frac14 \upgamma$$ then  one gets that 
\begin{align*}\forall \varphi\in& \mcC^\infty_c(M\setminus K_0)\colon\\
& \frac{\upgamma}{4}\left(\int_{M\setminus K_0} \varphi^{\frac{2n}{n-2}}\dv_{\bg}\right)^{1-\frac 2n}\le \int_{M\setminus K_0} \left[|d\varphi|_{\bg}^2+b \frac{ u^{-\frac{2}{\lambda}}|du|^2_g}{u^2} \varphi^2\right]\ \dv_{\bg}.\end{align*}
where $b=\frac{n-2}{2}\left(1-\frac{n-1}{2\lambda}\right)$.

The formula for the Laplacian of $\bg$ is 
$$\Delta_{\bg}\Psi=u^{-\frac 2\lambda}\left(\Delta_g \Psi-\frac{n-2}{\lambda}\la \log u,\Psi\ra_g\right)$$
Hence for $\Psi=u^\beta$ one gets
$$\Delta_{\bg}u^\beta=u^{-\frac 2\lambda}\left(\beta\lambda \Ricm u^\beta+\beta(1-\beta)|du|^2 u^{\beta-2}-\beta\frac{n-2}{\lambda}du|^2 u^{\beta-2}\right)$$
Chosing $\beta=\frac12 \left(1-\frac{n-1}{\lambda}\right)$, one gets 
$$\Delta_{\bg}u^\beta\ge \beta^2\frac{ u^{-\frac 2\lambda} |du|^2_g}{u^2}\ u^\beta,$$
And this implies that (see \pref{Barta}):
$$\forall \varphi\in \mcC^\infty_c(M)\colon  \beta^2 \int_{M}\frac{ u^{-\frac 2\lambda} |du|^2_g}{u^2}\varphi^2\ \dv_{\bg}\le  \int_M |d\varphi|_{\bg}^2\dv_{\bg}.$$
Letting $a=  1+b\beta^{-2}$, one gets that for  any $\varphi\in \mcC^\infty_c(M\setminus K)$:
$$
\frac{\upgamma}{4a}\left(\int_{M\setminus K} \varphi^{\frac{2n}{n-2}}\dv_{g}\right)^{1-\frac 2n}\le 
\int_M |d\varphi|_{\bg}^2\dv_{\bg}.$$
And  this implies the Sobolev inequality holds globally on $(M,\bg)$ (see \cite[Proposition 2.5]{Carron_1998}).
\endproof
\section{A Vanishing result for the $\upalpha$-genus}
\subsection{}This section is devoted to the proof of the following improvement of \cite{Carron_2002}:
\begin{theo}\label{Vanishing} Let $(M^n,g)$ be a complete Riemannian spin manifold $(M^n,g)$ with closed boundary $\partial M$ satisfying the Sobolev inequality \eqref{Sobn} and with non-negative scalar curvature 
$$\scal_g\ge 0$$ then if $\varsigma$ is the induced spin structure on $\partial M$:
$$\upalpha(M,\varsigma)=0.$$
\end{theo}
\begin{rems}
\begin{enumerate}[i)]
\item Let us clarify here what we mean by  $(M,g)$ being a complete Riemannian manifold with closed boundary:  introducing the geodesic distance $d_g$, complete  mean that $M$ is connected and that for all $R\ge 0$, 
 the sets $$M_R=\{x\in M\cup \partial M, d_g(x,\partial M)\le R\}$$ are compact or equivalently that the geodesic distance is complete on $\overline{M}=M\cup \partial M$. It is also equivalent to the fact that there are a complete Riemannian manifold $(N,g_N)$ and a compact set $K$ such that $(M,g)$ is isometric to some connected component of $(N\setminus K,g_N)$. 
\item  The space $\mcC^\infty_c(M)$ is the space of function  $\phi\in \mcC^\infty(\overline{M})$ such that there is  $R>r>0$ such that 
$\supp\phi\subset M_R\setminus M_r$.  Elements of $\mcC^\infty_c(M)$ vanishes along $\partial M$. We can also introduce $\mcC^\infty_c(\overline{M})$ who is the space of smooth function  $\phi\in \mcC^\infty_c(\overline{M})$ such that there is some $R>0$ such that 
$\supp\phi\subset M_R$. According to \cite[Proposition 2.5]{Carron_1998}, the following two Sobolev inequalities are equivalent
$$\exists \upgamma>0\colon \forall \varphi\in \mcC^\infty_c(M)\colon \upgamma\left(\int_M \varphi^{\frac{2n}{n-2}}\dv_g\right)^{1-\frac 2n}\le \int_M |d\varphi|^2\dv_g$$
$$\exists \upgamma'>0\colon \forall \varphi\in \mcC^\infty_c(\overline{M})\colon \upgamma'\left(\int_M \varphi^{\frac{2n}{n-2}}\dv_g\right)^{1-\frac 2n}\le \int_M |d\varphi|^2\dv_g.$$
\item Recall that if $(N,g,\varsigma)$ is a closed Spin Riemannian manifold and if $\Dirb\colon \mcC^{\infty}(N,\Spin)\to \mcC^{\infty}(N,\Spin)$ is its Dirac operator then \cite[Theorem 7.10]{LM}:
$$\upalpha(M,\varsigma)=\begin{cases}
\ind \Dirb^+ &\text{ if } n=8k\\
\frac12 \ind \Dirb^+&\text{ if } n=8k+4\\
\dim \ker  \Dirb^+\,(\text{mod }2)&\text{ if } n=8k+2\\
\dim \ker \Dirb\;\;\;(\text{mod }2)&\text{ if } n=8k+1.\end{cases}$$
\end{enumerate}
\end{rems}

This results has the following interesting topological corollaries:
\begin{coro}\label{3D2D} Let $(M^3,g)$ be a complete Riemannian spin manifold with closed and connected boundary $\partial M$ satisfying the Sobolev inequality \eqref{Sob} and with non-negative scalar curvature. If $H^1(M,\Z_2)\rightarrow H^1(\partial M,\Z_2)\rightarrow \{0\}$ then $\partial M$ is diffeomorphic to $\bS^2$.
\end{coro}
 \begin{coro}\label{simpleconnexe}Let $(M^3,g)$ be a complete Riemannian spin manifold $(M^n,g)$ without boundary satisfying the Sobolev inequality \eqref{Sob} and with non-negative scalar curvature. If   $\pi_2(M)=\{0\}$ then $M$ is simply connected.
\end{coro}
\subsection{Proof of the corollaries} Before embarquing in the proof of \tref{Vanishing}, let's explain how it implies these two corollaries:
\begin{proof}[Proof of \cref{3D2D}]Assume that $\partial M$ is not diffeomorphic to $\bS^2$. According to Atiyah \cite{Atiyah_1971}, then $\partial M$ has a spin structure $\varsigma$ with $\upalpha(\partial M,\varsigma)=1(\text{mod }2)$. Moreover $H^1(M,\Z_2)$ acts transitively on the set of spin structure of $M$ and  $H^1(\partial M,\Z_2)$ acts transitively on the set of spin structure of $\partial M$. Hence the surjectivity of the map  $H^1(M,\Z_2)\to H^1(\partial M,\Z_2)$  implies that  there is a spin structure on $M$ for which the induced spin structure is $\varsigma$ on $\partial M$. Recalling that $\upalpha(\partial M,\varsigma)=1(\text{mod }2)$, we see that \tref{Vanishing} implies that either the Sobolev inequality \eqref{Sobn} is not satisfied or the scalar curvature of $g$ is negative somewhere. 
\end{proof}
\begin{proof}[Proof of \cref{simpleconnexe}] We assume that $(M^3,g)$ is a complete Riemannian spin manifold $(M^n,g)$ without boundary satisfying the Sobolev inequality \eqref{Sob}, with non-negative scalar curvature and that $\pi_2(M)=\{0\}$. The universal cover of $M$ is contractible. And moreover we know that every non trivial element in $\pi_1(M)$ has infinite order (see \cite[Proof lemma 3]{Schoen_1982}). Assume by contradiction that we can find a non trivial $\gamma\in \pi_1(M)$. And let $\Gamma$ the infinite cyclic group generated by $\gamma$ and let $\widehat M=\widetilde{M}/\Gamma$ that we endowed with the Riemannian metric $\widehat g$ induced by $g$, the scalar curvature of $\widehat g$ is non negative and according to \pref{revsob}, $(\widehat M,\widehat g)$ satisfies the Sobolev inequality \eqref{Sob}.

 We know that $\widehat M$ is a $\text{K}(\Z,1)$, so that 
$H^2(\widehat M,\Z_2)=\{0\}$. Let $c\subset \widehat M$ a loop that generates $\pi_1(\widehat M)=\Z$. And let $\cU\subset M$ be a tubular neighborhood  of $c$. So that $\cU\simeq \bD\times \bS^1$ and $\partial U\simeq \bS^1\times \bS^1$. By construction the map $H^1(\widehat M,\Z_2)\to H^1(\cU,\Z_2)$ is an isomorphism hence the Mayer-Vietoris exact sequence 
$$H^1(\widehat M,\Z_2)\to H^1(\cU,\Z_2)\oplus H^1(\widehat M\setminus \cU,\Z_2)\to H^1(\partial\cU,\Z_2) \to H^2(\widehat M,\Z_2)=\{0\}.$$ implies that 
$$H^1(\widehat M\setminus \cU,\Z_2)\to H^1(\partial\cU,\Z_2) \to \{0\}$$ and this is a contracdiction with \cref{3D2D}.\end{proof}
\subsection{Proof of  \tref{Vanishing}}
We assume here that $(M^n,g)$ is a complete Riemannian spin manifold $(M^n,g)$ with closed boundary $\partial M$ satisfying the Sobolev inequality \eqref{Sobn} and with non-negative scalar curvature.
\subsubsection{Construction of the Dirichlet to Neuman map} We start by recalling the main construction of \cite{Carron_2001} in our setting. Let $\Spn\rightarrow M$ be the spin bundle over $M$ and $$\Dirac\colon \mcC^\infty(M,\Spn)\rightarrow \mcC^\infty(M,\Spn)$$ be the associated Dirac operator. We let $\nabla$ be the Levi-Civita connexion or the induced connexion on $\Spn$. 
As the scalar curvature of $g$ is non negative, the Dirac operator $\Dirac$ is non parabolic at infinity. If one considers a compact set $K\subset \overline{M}$ containing $\partial M$ and with non empty interior, we define $\mcW(M,\Spn)$ who is the completion of the space $\mcC_c^\infty(\overline{M},\Spn)$ with respect to the norm $$\sigma\mapsto \sqrt{ \| \Dirac\sigma\|_{L^2}^2+\| \sigma\|_{L^2(K)}^2}.$$
 
This space does not depend on the choice of $K$ and there is a natural continuous map
$$\mcW(M,\Spn)\to \mcW^{1,2}_{loc}(\overline{M},\Spn). $$ There is also a continuous trace map
$$\tr\colon \mcW(M,\Spn)\to H^{1/2}(\partial M,\Spn),$$ so that an equivalent norm on $\mcW(M,\Spn)$ is 
$$\sigma\mapsto \sqrt{ \| \Dirac\sigma\|_{L^2}^2+\| \tr(\sigma)\|_{H^{1/2}}^2}=\|\sigma\|_{\mcW}.$$
\begin{prop} If $\sigma\in H^{1/2}(\partial M,\Spn)$ then there is a unique $h(\sigma)\in \mcW(M,\Spn)$ such that 
$$\tr(h(\sigma))=\sigma\text{ and } \Dirac^2(h(\sigma))=0.$$
Moreover if $\sigma\in \mcC^\infty(\partial M,\Spn)$ then $h(\sigma)\in  \mcC^\infty(\overline{M},\Spn)$.
\end{prop}
\begin{defi} We introduce a map $T\colon \mcC^\infty(\partial M,\Spn)\to \mcC^\infty(\partial M,\Spn)$ by
$$T(\sigma)=\tr\left(\Dirac h(\sigma)\right).$$
\end{defi}
The main properties of this Dirichlet to Neumann map are the following:
\begin{prop}\begin{enumerate}[i)]
\item $T$ is a pseudo-differential operator of order $1$
\item  $T\circ T=0,$
\item We denote Clifford multiplication by $\cl$ then if $\vec n\colon \partial M\rightarrow TM$ in the unit inward normal vector field, then $\cl(\vec n)T$ is selfadjoint and $T^*=\cl(\vec n)T\cl(\vec n).$
\item The cohomology defined by $T$ has the following interpretation:
$$\dim\frac{\ker T}{\text{Im } T}=\dim\ker\left(T+T^*\right)=\dim\frac{\left\{\sigma\in \mcW(M,\Spn)\cap\mcC^\infty(\overline{M}, \Spn)\colon \Dirac \sigma=0\right\}}{\left\{\sigma\in L^2(M,\Spn)\cap\mcC^\infty(\overline{M}, \Spn)\colon \Dirac \sigma=0\right\}}.$$
\end{enumerate}
\end{prop}

\subsubsection{A vanishing result for the cohomology defined by  the Dirichlet to Neuman map} 
The following is the desired improvement of \cite{Carron_2002}:
\begin{prop}\label{annulation}Under the hypothesis of \tref{Vanishing}, we have $\ker\left(T+T^*\right)=\{0\}.$
\end{prop}
\proof It is here useful to recall how the harmonic extension of $\sigma\in  H^{1/2}(\partial M,\Spn)$ is constructed.
If we define $\mcW_0(M,\Spn)$ to be the kernel of the trace map $\tr$, then $\mcC_c^\infty(M,\Spn)$ is dense in $\mcW_0(M,\Spn)$. If $\tilde \sigma\in \mcW(M,\Spn)$ is any extension of $\sigma$ that has compact support in $\overline{M}$; then $\tilde\sigma-h(\sigma)$ is the $\mcW$-orthogonal projection of $\tilde\sigma$ on $\mcW_0(M,\Spn)$.
The Lichnerowicz formula is that for $\phi\in \mcC_c^\infty(M,\Spn)$:
$$\int_M |\Dirac \phi|^2\dv_g=\int_M |\nabla \phi|^2\dv_g+\frac14\int_M \scal |\phi|^2\dv_g\ge \int_M |\nabla \phi|^2\dv_g.$$
So that with the Kato inequality $|\nabla \phi|\ge |\nabla |\phi|\, |$ and the Sobolev inequality \ref{Sobn},  one gets that for any $\phi\in \mcC_c^\infty(M,\Spn)$:
$$\int_M |\Dirac \phi|^2\dv_g\ge \upgamma \left(\int_M |\phi|^{\frac{2n}{n-2}}\dv_g\right)^{1-\frac2n}.$$
This implies that for any $\phi\in \mcW_0(M,\Spn)$, one gets $\phi\in L^{\frac{2n}{n-2}}.$ In particular 
$h(\sigma)\in  L^{\frac{2n}{n-2}}.$
The Kato inequality and the fact that $$\Dirac^2h(\sigma)=\nabla^*\nabla h(\sigma)+\frac14 \scal_g\; h(\sigma)=0$$ yields that 
$u=|h(\sigma)|$ is subharmonic:
$$\Delta_g u\le 0.$$
The Sobolev inequality and the classical De Giorgi–NashMoser iteration method iteration scheme yields that for any 
$x\in M $ such that $d(x,\partial M)\ge R:$
\begin{equation}\label{tendvers0}u(x)\le C(n,\upgamma) \frac{1}{R^{\frac{n-2}{2}}}\  \left(\int_{B(x,R)} u^{\frac{2n}{n-2}}\dv_g\right)^{\frac{n-2}{2n}}.\end{equation}
In order to prove the \pref{annulation}, we need to show that if $\Dirac h(\sigma)=0$ then $h(\sigma)$ is in $L^2$. 

Hence let's consider $\sigma\in H^{1/2}(\partial M,\Spn)$ such that $\Dirac h(\sigma)=0$. The refined Kato inequality (\cite[Theorem 4]{Herzlich} )yields that for $f:=|h(\sigma)|^{\frac{n-2}{n-1}}$ one gets 
$$\Delta f+\frac{n-2}{4(n-1)}\scal_g\;f\le 0.$$  Hence $f$ is subharmonic. Moreover by \eqref{tendvers0}, one knows that $\lim_{x\to \infty} f(x)=0$.
We introduce now $M_R=\{x\in \overline{M}, d(x,\partial M)\le R\}$ and $h\colon M_1\to (0,1]$ the equilibrium potential of $M_1$, i.e. $h$ is the monotone limit of the Dirichlet boundary value problem
$$\begin{cases}
\Delta h_R=0&\text{ on } M_1\setminus M_R\\
h_R=1&\text{ on } \partial M_1\\
h_R=0&\text{ on } \partial M_R.
\end{cases}
$$
According to \eqref{integrapot}, one has that $h\in L^{p}$ for any
$p>n/(n-2)$.
The maximum principle yields that if $L_R=\| h(\sigma)\|_{L^\infty(\partial M_R)}$ then
$$\forall x\in M_R\colon f\le L_1 h_R+L_R.$$
Letting $R\to+\infty$, one obtains that 
$$\forall x\in M_R\colon 0\le f\le L_1 h$$ so that for any $p>n/(n-2)$, one has that $f\in L^{p}$. And eventualy one gets that 
$$\forall p>n/(n-1)\colon |h(\sigma)|\in L^p(M_1).$$
In particular, $|h(\sigma)|\in L^2$ which is the desired conclusion.
\endproof
\subsubsection{Conclusion via Callias index formula} 
The principal symbol of the pseudo-differential operator $T$ is given by
$$\sigma(T)(x,\xi)=-|\xi|\cl(\vec n)+i\cl(\xi).$$ Hence the principal symbol of $T+T^*$ is given by
\begin{equation}\label{prsim}
\sigma(T+T^*)(x,\xi)=2i\cl(\xi).\end{equation}
Moreover $T+T^*=T+\cl(\vec n)T\cl(\vec n)$ anticommutes with the operator of Clifford multiplication by $\vec n$:
$$\left(T+T^*\right)\cl(\vec n)=-\cl(\vec n)\left(T+T^*\right).$$
In order to finish the proof we need to discuss the different case in function of the dimension of $M$ that is when $n=4k+1, 8k+3,8k+1$. We use here the exposition of \cite[Section 2]{Bar_GAFA} and \cite[Subsection 1.7]{FriedrichDirac}
\\
{\bf Case where the dimension of $M$ is odd.} Then if one endows $\partial M$ with the spin structure induced form the one of $M$, we know that the spin bundle of $\Spin\to \partial M$ can be identified with $\Spn\to \partial M$. Moreover its $\Z_2$ grading is given by the eigenspace of $\cl(\vec n)$, for any $x\in \partial M$:
$$\Spin_x^\pm=\left\{\sigma\in \Spn_x\ ;\ \cl(\vec n(x))\sigma=\pm \sigma\right\}.$$
As $T+T^*$ anticommutes with the operator of Clifford multiplication by $\vec n$, one get that
$$(T+T^*)\left( \mcC^\infty\left(\partial M,\Spin^\pm\right)\right)\subset  \mcC^\infty\left(\partial M,\Spin^\mp\right).$$
Noticed that the Clifford multiplication by $\xi\in T_x\partial M$ on $\Spin_x$ is given by
$$\cl_{\partial M}(\xi)=\cl(\vec n(x))\cl(\xi).$$
Hence the Dirac operator $\Dirb^+\colon \mcC^\infty\left(\partial M,\Spin^+\right)\to \mcC^\infty\left(\partial M,\Spin^-\right)$ and 
the pseudo-differential operator $T+T^*\colon \mcC^\infty\left(\partial M,\Spin^+\right)\to \mcC^\infty\left(\partial M,\Spin^-\right)$ have proportional principal symbol, hence then have the same index hence $\ind \Dirb^+=0$ and this conclude the proof of the theorem when $n=4k+1$.
\\
{\bf When the dimension of $M$ is $8k+3$.}  In that case $\Spn$ can be endowed with a parallel real endomorphism $J$, who is anti-$\C$ linear and commutes with the Clifford multiplication and satisfying $J^2=-\text{Id}$. Hence $J\Spin^\pm=\Spin^\mp.$ Moreover being parallel, $J$ commute with the harmonic extension $h(J\sigma)=Jh(\sigma)$ and anticommutes with $T$ and with $T+T^*$. Similarly the operator $J(T+T^*),\ J\Dirb\colon  \mcC^\infty \left(\partial M,\Spin^+\right)\rightarrow  \mcC^\infty\left(\partial M,\Spin^+\right)$ are anti-$\C$ linear and skew-symmetric and elliptic, and they have proportional  principal symbol hence the same $(\text{mod } 2)$-index that is zero by \pref{annulation}.  So that 
$$\alpha(\partial M,\varsigma)=\dim\ker J\Dirb (\text{mod } 2)=0\ (\text{mod } 2).$$
\\
{\bf When the dimension of $M$ is $8k+2$.} Then if one endows $\partial M$ with the spin structure induced form the one of $M$, we know that the spin bundle of $\Spin\to \partial M$ can be identified with $\Spn^+\to \partial M$. Moreover $\Spn$ can be endowed with a parallel real endomorphism $\widehat J$ who is anti-$\C$ linear and anti-commutes with the Clifford multiplication and satisfying $\widehat J^2=\text{Id}$. Because $\Spn^\pm$ are the $\pm i$-eigenspace of the Clifford multiplication by the volume form, we have $\widehat J\Spn^\pm=\Spn^\mp.$ Then the operators $$\cl(\vec \nu)\widehat J\Dirb,\ \widehat J(T+T^*)\colon \rightarrow  \mcC^\infty(\partial M,\Spin) \rightarrow  \mcC^\infty(\partial M,\Spin)$$ are also anti-$\C$ linear and skew-symmetric and elliptic, and they have proportional  principal symbol hence the same $(\text{mod } 2)$-index that is zero by \pref{annulation}.  So that 
$$\alpha(\partial M,\varsigma)=\dim\ker  \cl(\vec \nu)\widehat J\Dirb^+\ (\text{mod } 2)=0\ (\text{mod } 2).$$

\section{Topology of the level set of the Green function}
In this section, we will use a variational formula for area-type functionals associated with the level set of harmonic functions in order to obtain certain topological constraints on complete Riemannian $3-$manifolds satsifying the Sobolev inequality and a $L^{3/2}$ integrability of the negative part of the scalar curvature.
\begin{prop}\label{levelgreen} Let $(M^3,g)$ be a complete Riemannian manifold satisfying the Sobolev inequality \eqref{Sob} and such that 
\begin{enumerate}[i)]
\item the negative part of the scalar curvature satisfies $\scal_-\in L^{3/2}$,
\item for some $o\in M$ and if $G_o$ is the Green kernel with pôle at $o$ we have $$\int_{M\setminus B(o,1)} \frac{ |dG_o|^3}{G_o^2}\dv_g<\infty,$$
\item The regular level set of $G_o$ are  compact,
\item The regular level set of $G_o$ are  connected.
\end{enumerate}
Then there is a sequence $\epsilon_k\to 0$ such that for $\Omega_k=\{G_o>\epsilon_k\}$ has spherical boundary: $\partial \Omega_k\simeq \bS^2.$
\end{prop}
\begin{rem}
It should be noticed that the fourth hypothesis is satisfied as soon as 
$$\text{Im}\left( H_c^1(M,\R)\to H^1(M)\right)=\{0\}$$ that is, as soon as the closed, compactly supported forms of degree 1 are exact.
Indeed assume that there is some $t>0$  such that  $\Sigma=\{G_o=t\}$ is smooth and not connected. The maximum principle says that the unique bounded
connected component of $M\setminus \Sigma$  is the one containing the pole $o$.
This topological hypothesis implies that $M$ has only one end so that $M\setminus \Sigma$ has two connected components that are $\{G_o<t\}$ and $\{G_o>t\}$. If one considers a smooth function $v$ defined on $\{G_o>t-\varepsilon\}$ such that it is constant equals to one in a connected component of $\{G_o\in [t-\varepsilon,t]\}$ and zero on the others then $dv$ has compact support but can not be exact.
\end{rem}
\begin{rem} It should also be noticed that the Sobolev inequality \eqref{Sob}, implies that the Green kernel tends to zero at infinity hence its positive level set are compact.
\end{rem}
\begin{proof}{Proof of \pref{levelgreen}}
We introduce the function, $b$, the fake distance to $o$, by $b(o)=0$ and for any $x\in M\setminus\{o\}:$
$$G_o(x)=G(o,x)=\frac{1}{4\pi b(x)}.$$
 The Sobolev inequality implies that the fake ball $\cU_r=\{b<r\}$ are bounded and the estimate \eqref{volfakeball} provides a Euclidean upper bound on the volume of the fake ball $\cU_r=\{b<r\}$:
 \begin{equation}\label{volfakeball2}
 \vol_g\left( \cU_r\right)\le C(\upgamma) r^3.\end{equation}
 By Sard's theorem, we know that for a.e. $r>0$, the level set $\partial \cU_r$ is smooth and we introduce the aera type functional:
 $$A(r)=\int_{\partial \cU_r} \frac{ |db|_2^2}{r^2}\dA_g=\int_{\partial \cU_r} \frac{ |db|_2^2}{b^2}\dA_g.$$
 This is defined a priori for a.e. $r>0$ but we know that $A$ has a $\mcC^1$ extension to $(0,+\infty)$ (This is discuss in \cite{Colding:2025aa}).
 Noticed that the second hypothesis in the proposition is that 
 \begin{equation}\label{upperA}\int_1^{+\infty} \frac{A(s)}{s^2} ds<+\infty.\end{equation}
 This hypothesis says that $A$ is not too large at infinity. Let's now explain why  $A$ can not be too small, this is a consequence of the Sobolev inequality and of the Green formula:
 $$\int_{\partial \cU_r} \frac{|db|_g}{4\pi b^2}\dA_g=\int_{\partial \cU_r} \frac{\partial G}{\partial \vec\nu}\dA_g=\int_{\cU_r} \Delta_g G \dv_g=1$$
With the co-aera formula, one has that:
\begin{align*}
4\pi r&=\int_{r}^{2r} \left(\int_{\partial \cU_s} \frac{|db|_g}{ b^2}\dA_g\right)ds\\
&=\int_{\cU_{2r}\setminus \cU_r} \frac{|db|_g^2}{b^2}\dv_g\\
&\le \left(\int_{\cU_{2r}\setminus \cU_r} \frac{1}{b^2}\dv_g\right)^{\frac13}\left(\int_{\cU_{2r}\setminus \cU_r} \frac{|db|_g^3}{b^2}\dv_g\right)^{\frac23}\\
&\le \frac{\left(\vol_g(\cU_{2r})\right)^{\frac13}}{r^{\frac23}}\, \left(\int_{r}^{2r} A(s)ds\right)^{\frac23}.
\end{align*}
 Hence using \eqref{volfakeball2}, we can conclude that there is a positive constant $c(\upgamma)$ depending only on $\upgamma$ such that for any $r>0$:
 \begin{equation}\label{lowerA}
 c(\upgamma) \le \frac{1}{r} \int_{r}^{2r} A(s)ds.
 \end{equation}
According to Colding-Minicozzi \cite[formula 1.33 and Lemma 2.1]{Colding:2025aa} and the coaera formula, for any $0<r<R$ one gets the inequality:
\begin{align*}A'(R)+\frac{1}{R}A(R)&-A'(r)-\frac{1}{r}A(r)\ge\\
& \int_{\cU_R\setminus \cU_r}\!\!\! \scal_g\frac{\,|db|_g}{2b^2}\dv_g+\frac34\int_r^R \frac{\left(A'(s)\right)^2}{A(s)}ds-\int_r^R\frac{4\pi \chi(\partial \cU_s)}{s^2}ds.\end{align*}
Where $\chi(\partial \cU_s)$ is the Euler characteristic of the surface $\partial \cU_s$. If one assumes that the conclusion of \pref{levelgreen} is not true, there is some $r_0$ such that for any $s\ge r_0:$ $$ \chi(\partial \cU_s)\le 0.$$
Hence one gets that for any $r_0<r<R$:
 \begin{equation}\label{ineqfund}A'(R)+\frac{1}{R}A(R)-A'(r)-\frac{1}{r}A(r)\ge -\frac12\int_{\cU_R\setminus \cU_r}\!\!\! (\scal_g)_-\frac{|db|_g}{b^2}\dv_g+\frac34\int_r^R \frac{\left(A'(s)\right)^2}{A(s)}ds.\end{equation}
We use the Hölder inequality to estimate the first term in the LHS of \eqref{ineqfund} and gets that for any $\delta>0:$
\begin{align}\label{firstestimate}
\int_{\cU_R\setminus \cU_r}\!\!\! \frac{(\scal_g)_-\,|db|_g}{b^2}\dv_g&\le \left(\int_{\cU_R\setminus \cU_r}\!\!\! \frac{(\scal_g)^{\frac32}_-}{b}\dv_g\right)^{\frac23}\left(\int_r^R\frac{A(s)}{s^2}ds\right)^{\frac13}\nonumber\\
&\le \frac{2}{3\sqrt{\delta}}\int_{\cU_R\setminus \cU_r}\!\!\! \frac{(\scal_g)^{\frac32}_-}{b}\dv_g+\frac{\delta}{3}\int_r^R\frac{A(s)}{s^2}ds.
\end{align}
In order to deal with the second term in the LHS of  \eqref{ineqfund}, we use the Hardy type inequality:
$$\forall u\in \mcC^1([a,b])\colon \int_a^b |u'|^2 ds\ge \frac14 \int \frac{u^2}{s^2} ds+\frac{u^2(b)}{2b}-\frac{u^2(a)}{2a}.$$
Which is proven by develloping $\int_a^b \left| u'-\frac{1}{2r} u\right|^2dr$ and integrating by part the term
$$\int_a^b \frac{u(s)u'(s)}{s}ds=\frac{u^2(b)}{2b}-\frac{u^2(a)}{2a}+\frac12  \int_a^b \frac{u^2(s)}{s^2} ds.$$
Hence one gets that for any $\mu\in [0,3/4]$:
\begin{equation}\label{secondestimate}\frac34\int_r^R \frac{\left(A'(s)\right)^2}{A(s)}ds=3\int_r^R\left(\sqrt{A(s)}\,'\right)^2ds\ge \frac{\mu}{4}\int_r^R \frac{A(s)}{s^2}ds+\frac{\mu}{2}\left(\frac{A(R)}{R}-\frac{A(r)}{r}\right).\end{equation}
We choose $\delta=\frac32$ and $\mu=1$ and inserting \eqref{firstestimate} and \eqref{secondestimate} in \eqref{ineqfund}, one obtains\footnote{The constant $1/2$ is not the optimal one, we chose it for convenience.}that for any $r_0<r<R$:
$$A'(R)+\frac{1}{2R}A(R)-A'(r)-\frac{1}{2r}A(r)\ge -\frac12\int_{\cU_R\setminus \cU_r} \!\!\! \frac{(\scal_g)^{\frac32}_-}{b}\dv_g.$$
The condition \eqref{upperA} implies that 
$$\liminf_{R\to+\infty} A'(R)+\frac{1}{2R}A(R)\le 0$$ hence one has that for any $r\ge r_0:$
$$\frac{1}{\sqrt{r}} \frac{d}{dr} \left(\sqrt{r} A(r)\right)=A'(r)+\frac{1}{2r}A(r)\le \frac12\int_{M\setminus \cU_r} \!\!\! \frac{(\scal_g)^{\frac32}_-}{b}\dv_g\le \frac{1}{2r} \int_{M\setminus \cU_r} \!\!\! (\scal_g)^{\frac32}_-\dv_g .$$
So that for any $R>r>r_0$:
$$\sqrt{R}A(R)-\sqrt{r}A(r)\le \left(\sqrt{R}-\sqrt{r}\right) \int_{M\setminus \cU_r} \!\!\! (\scal_g)^{\frac32}_-\dv_g.$$
Hence one get that for any $r>r_0$:
$$\limsup_{R\to+\infty} A(R)\le  \int_{M\setminus \cU_r} \!\!\! (\scal_g)^{\frac32}_-\dv_g$$ so that
$\limsup_{R\to+\infty}  A(R)=0$ and  this is a contradiction with \eqref{lowerA}.

\end{proof}
\section{Proof of the theorems}
In order to prove \tref{3D-neg} and \tref{3D-full}, we proceed according to the proof scheme introduced by Schoen and Yau and demonstrate successively that the universal cover is contractible, then we explain why it is simply connected and eventually we show that the manifold is simlply connected at infinity.
\subsection{The universal cover is contractible}
This first step relies on the following result:
\begin{prop}\label{prop-noends} Let $(M^n,g)$ be complete Riemannian manifold satisfying the Sobolev inequality \eqref{Sobn} and such that the Schrödinger operator $\Delta_g+\Ric$ is non negative then $M$ and any of its normal covering $\widehat M\rightarrow M$ has only one end, as a consequence $H^{n-1}(\widehat M,\Z)=\{0\}$.
\end{prop}
\proof Let $\pi\colon \widehat M\rightarrow M$ be a normal covering of $M$, according to \pref{revsob}, we know that if $\widehat g=\pi^*g$ then $(\widehat M,\widehat g)$ satisfies also the Sobolev inequality \eqref{Sobn} and that the Schrödinger operator $\Delta_{\widehat g}+\Ric_{\widehat g}$ is also non negative. If $\widehat M$ has two ends, the Sobolev inequality implies that we can find an non constant  bounded harmonic function $h\colon \widehat M\rightarrow \R$ with finite Dirichlet energy \cite[Theorem 3]{Carron_1998} or \cite[Proof of Lemma 2]{CaoShenZhu}:
$$\int_{\widehat M} |dh|_{\widehat g}^2\dv_{\widehat g}<+\infty.$$
Let $\alpha=dh$, using the Bochner formula and integrating by part, we get that for any $\chi\in \mcC_c^\infty(\widehat M):$
\begin{align*}
\int_{\widehat M} \left[ \ |\nabla \chi \alpha|^2+\Ricci(\chi\alpha,\chi\alpha)\right] \dv_{\widehat g}&=\int_{\widehat M}|d\chi|^2|\alpha|^2 \dv_{\widehat g}\\
&\hspace{1cm}+\int_{\widehat M} \chi^2 \left[ \ \la \nabla^*\nabla\alpha+\Ricci(\alpha), \alpha\ra\right] \dv_{\widehat g}\\
&=\int_{\widehat M}|d\chi|^2|\alpha|^2 \dv_{\widehat g}.
\end{align*}
Recall the refined Kato inequality of Yau:
$$|\nabla \alpha|^2\ge \frac{n}{n-1} | d|\alpha|\,|^2.$$
Now for every $\ve\in (0,1)$ one gets that 
$$|\nabla \chi \alpha|^2\ge (1-\ve)\chi ^2 |\nabla \alpha|^2+\left(1-1/\ve\right) |\alpha|^2 |d\chi|^2$$ and
$$\frac{1}{\ve} |\alpha|^2 |d\chi|^2+ \chi ^2 |d| \alpha|\, |^2\ge \frac{1}{1+\ve}  |\nabla \chi|\alpha||^2.$$
Hence, we find a constant $C$ depending only on $n$ and $\ve$ such that 
$$ C|\alpha|^2 |d\chi|^2+ |\nabla \chi \alpha|^2\ge \frac{1-\ve }{1+\ve}\frac{n}{n-1}   |d \chi|\alpha||^2$$
We choose $\ve>0$ so that 
$$ \frac{1-\ve }{1+\ve}\frac{n}{n-1}=1+\frac1n$$ and get that 
\begin{align*}
(C+1)\int_{\widehat M}|d\chi|^2|\alpha|^2 \dv_{\widehat g}&\ge \int_{\widehat M} \left[C|\alpha|^2 |d\chi|^2 +|\nabla \chi \alpha|^2\right] \dv_{\widehat g}+\int_{\widehat M}\Ric \chi^2|\alpha|^2 \dv_{\widehat g}\\
&\ge \frac1n  \int_{\widehat M}  |d \chi|\alpha||^2\dv_{\widehat g}+\int_{\widehat M} \left[|d \chi|\alpha||^2+ \Ric \chi^2|\alpha|^2\right] \dv_{\widehat g}\\
&\ge \frac1n  \int_{\widehat M}  |d \chi|\alpha||^2\dv_{\widehat g}.
\end{align*}
Choosing now $\chi=\chi_R$ such that its gradient is bounded by $2$ and such that 
$$\chi_R(x)=\begin{cases} 0&\text{ on } B(x,R)\\
1&\text{ outside } B(x,R+1)\end{cases}$$ and letting $R\to+\infty$, we see that 
$\alpha=0$. Hence a contradiction so that $\widehat M$ has only one end.
\endproof
\begin{coro} Under the assumption of \tref{3D-neg} and \tref{3D-full}, the universal cover of $M$ is contractible.
\end{coro}
\proof Let $\pi \colon \tilde M\rightarrow M$ be the universal cover of $M$.  According \pref{prop-noends}, $H^2( \tilde M,\Z)=\{0\}$ and $ \tilde M$ being simply connected, we get that  $\pi_2(\tilde M)=\{0\}$. Hence $ \tilde M$ is contractible.
\endproof 
\subsection{Simple connectness}
\begin{prop} Under the assumption of \tref{3D-neg} or of \tref{3D-full}, $M$ is simply connected.
\end{prop}
\proof In the setting of \tref{3D-full}, we know that the Riemannian manifold $(M,\bg)$ is complete, satisfies a Sobolev inequality (\pref{Sobbarg})  and that it has non negative scalar curvature (\lref{scalarbarg}). Hence the result follows from an application of  \cref{simpleconnexe} and by  he fact that the universal cover of $M$ is contractible.

In the setting of \tref{3D-neg}, we use the geometry of the weighted Riemannian manifold $(M,\bg,\mu)$. Let $\pi \colon \tilde M\rightarrow M$ be the universal cover of $M$, we endowed it with the metric $\tilde g=\pi^*\bg$ and the measure $\tilde\mu=\pi^*\mu.$ The  weighted Riemannian manifold $(\tilde M,\tilde g,\tilde \mu)$ satisfies the Bakry-Emery condition $\BE(0,N)$ and the assumptions made in \tref{3D-neg} guarantees that $$N=3+\frac{1}{\lambda-1}<4.$$
We know that every non trivial element in $\pi_1(M)$ has infinite order. Hence assume that $\pi_1(M)$ is non-trivial and let $\gamma\in \pi_1(M)\setminus\{e\}$.
Now let's fix a point $\tilde o\in \tilde M$ and let $o=\pi(\tilde o)$. Let $\cF\subset \tilde M$ the interior of the Dirichlet fundamental domain anchored at $\tilde o:$
$$\cF=\{x\in \tilde M\colon d_{\tilde g}(\tilde o, x)< d_{\tilde g}(\tau(\tilde o), x)\, ;\, \forall \tau \in\pi_1(M) \}$$
Let $L=d_{\tilde g}(\tilde o, \gamma(\tilde o))$ so that for every $k\in \N$:
$$d_{\tilde g}(\tilde o, \gamma^k(\tilde o))\le kL.$$
Hence 
$$\bigcup_{\ell=0}^{k} \gamma^{\ell}\left(\cF\cap B_{\tilde g}(\tilde o,R)\right)\subset B_{\tilde g}(\tilde o,R+kL).$$
The Bishop-Gromov comparison theorem yields that if $R\ge 1$ then
\begin{equation}\label{uppervol}
\tilde \mu\left(B_{\tilde g}(\tilde o,R+kL)\right)\le \tilde \mu\left(B_{\tilde g}(\tilde o,1)\right)\ \left(R+kL\right)^{N}.\end{equation}

We also know  the sets $\gamma^{\ell}\left(\cF\cap B_{\tilde g}(\tilde o,R)\right)$ are pairwise disjoints and have the same $\tilde \mu$-volume. And according to Anderson \cite{Anderson_1990}, we know that for every $R>0$:
$$\pi\left(\overline{\cF}\cap B_{\tilde g}(\tilde o,R)\right)=B_{\bg}( o,R) \text{ and } \mu\left(\pi\left(\cF\cap B_{\tilde g}(\tilde o,R)\right)\right)=\mu\left(B_{\bg}( o,R)\right). $$
Now using the lower volume estimate (\pref{muvollower}), one knows that for some positive constant $C$ such that 
$$\forall R>0\colon C R^3\le\mu\left( B_{\bg}( o,R)\right).$$ 
Eventually, one obtains :
$$C (k+1)R^3\le \mu\left(B_{\tilde g}(\tilde o,1)\right)\ \left(R+kL\right)^{N}.$$
If one chooses $R=k$, we see that the fact that $N<4$ is in contradiction with the fact that $\gamma$ has infinite order.
\endproof\subsection{Simple connectness at infinity}
\begin{prop} Under the assumption of \tref{3D-neg} or of \tref{3D-full}, $M$ is simply connected at infinity.
\end{prop}
\proof  We noticed first that the assumption made in  both Theorems and \pref{Greengradient} implies that $(M,g)$ satisfies the first two hypothesis of \pref{levelgreen}.

As $M$ is simply connected and as $(M,g)$ satisfies the Sobolev inequality \eqref{Sob}, we known that $(M,g)$ satisfies the last two hypothesis of \pref{levelgreen}. Hence $M$ satisfies the the conclusion \pref{levelgreen}: there is a exhaustion  $$M=\cup_{\ell\in \N} \Omega_\ell$$ such that $\partial \Omega_\ell\simeq \bS^2$.
Van Kampen's theorem yields that each $M\setminus  \Omega_\ell$ is simply connected hence $M$ is  simply connected at infinity.
\endproof

\appendix
\section{Two technical results}
Here we record some results that have been used repeatedly in the paper.

\begin{prop}\label{Barta}Let $(M^n,g)$ be a complete Riemannian open manifold and let $V\in L^\infty_{loc}(M)$ then the following are equivalent:
\begin{enumerate}[i)]
\item The Schrödinger operator $\Delta_g+V$ is non negative: $$\forall \phi\in \mcC_c^\infty(M)\colon \int_M \left[ |d\phi|^2_g+V\phi^2\right|\;\dv_g\ge 0.$$
\item There is a positive function $u\in W^{2,2n}_{loc}(M)$, such that $\Delta_g u+Vu=0$.
\item There is a positive function $u\in \ W^{2,2n}_{loc}(M)$, such that $\Delta_g u+Vu\ge0$.
\end{enumerate}
\end{prop}
\begin{rems}
\begin{enumerate}[i)]
\item This result is stated in a slightly stronger form in \cite[Lemma 3.10]{pigola}
\item The implication iii)$\Rightarrow$ i) is known as the Barta's trick.
\item The most difficult implication is i)$\Rightarrow$ ii) ; it is due to Moss and Piepenbrink in a Euclidean setting \cite{MP} and to Fischer-Colbrie and Schoen in the Riemannian setting \cite{FS}.
\end{enumerate}
\end{rems}

\begin{prop}\label{revsob} Let $(M^n,g)$ be a complete Riemannian manifold and let $\pi\colon \widehat M\rightarrow M$ be a  covering endowed with the metric $\widehat g=\pi^*g$.
\begin{enumerate}[i)]
\item If  $W$ is locally bounded and if the Schrödinger operator $\Delta_g+W$ is non negative then the Schrödinger operator $\Delta_{\widehat g}-W\circ\pi$ is also non negative.
\item If $n>2$ and if $(M^n,g)$ satisfies the Sobolev inequality \eqref{Sobn}, then $(\widehat M,\widehat g)$ satisfies the same Sobolev inequality (with the same Sobolev constant $\upgamma$):
$$\forall \varphi\in \mcC^\infty_c(\widehat M)\colon \upgamma\left(\int_{\widehat M} \varphi^{\frac{2n}{n-2}}\dv_{\widehat g}\right)^{1-\frac 2n}\le \int_{\widehat M} |d\varphi|_{\widehat g}^2\dv_{\widehat g}.$$\end{enumerate}
\end{prop}
\proof
For the first point, the assumption implies that we can find some positive solution $u\in  W^{2,2n}_{loc}(M)$  of the equation $\Delta_g u=Vu$ and then $v=u\circ\pi$ is also a positive solution of the equation $\Delta_{\widehat g} v=V\circ \pi v$. And with \pref{Barta}, we know that the  Schrödinger operator $\Delta_{\widehat g}-V\circ\pi$ is non negative.

Concerning the second point, with a $L^{n/2}-L^{n/(n-2)}$ duality's argument, it is easy to show that the Sobolev inequality \eqref{Sobn} is equivalent to
the fact that for any  non negative  $V\in \mcC_c^0(M)$ we have that 
$$\forall \varphi\in \mcC^\infty_c( M)\colon\int_M  V\varphi^{2}\dv_{g}\le \upgamma^{-1} \| V\|_{L^{\frac n 2}(M)} \int_M |d\varphi|_{g}^2\dv_{ g}.$$
Now let $V\in \mcC_c^0(\widehat M)$ being non negative and define 
$\widehat Q$ by
$$\widehat Q(x)=\left(\sum_{y\in\pi^{-1}\{\pi(x)\} } V^{\frac n2}(y)\right)^{\frac 2n}.$$
By construction, there is  $Q\in \mcC^0_c( M)$ such that $\widehat Q=Q\circ \pi$ and moreover
$$\| Q\|_{L^{\frac n 2}(M)}=\| V\|_{L^{\frac n 2}(\widehat M)}.$$
The Sobolev inequality on $(M,g)$ implies that
$$\forall \varphi\in \mcC^\infty_c( M)\colon\int_M  Q\varphi^{2}\dv_{g}\le \upgamma^{-1} \| V\|_{L^{\frac n 2}(\widehat M)} \int_M |d\varphi|_{g}^2\dv_{ g}.$$
Let $$W=\upgamma \| V\|^{-1}_{L^{\frac n 2}(\widehat M)} Q,$$ then one gets that the   Schrödinger operator $\Delta_{ g}-W$ is non negative. The first step implies that the  Schrödinger operator $\Delta_{\widehat g}-W\circ\pi$ is non negative, that is to say
$$\forall \varphi\in \mcC^\infty_c( \widehat M)\colon\int_M  \widehat Q \varphi^{2}\dv_{\widehat g}\le \upgamma^{-1} \| V\|_{L^{\frac n 2}(\widehat M)} \int_M |d\varphi|_{\widehat g}^2\dv_{\widehat g}.$$
The fact that $V\le \widehat Q $ implies then that 

$$\forall \varphi\in \mcC^\infty_c( \widehat M)\colon\int_M  V \varphi^{2}\dv_{\widehat g}\le \upgamma^{-1} \| V\|_{L^{\frac n 2}(\widehat M)} \int_M |d\varphi|_{\widehat g}^2\dv_{\widehat g}.$$
And the Sobolev inequality holds on $(\widehat M,\widehat g)$ with the same Sobolev constant $\upgamma$.

\endproof

\bibliographystyle{alpha} 
\bibliography{biblio.bib}
\end{document}